\documentclass[12pt]{amsart}
\usepackage{amsmath,amsthm,amsfonts,amssymb,mathrsfs}
\usepackage{pb-diagram}
\date{\today}

\usepackage{color}
\input xymatrix
\xyoption{all}

\usepackage{hyperref}

\newtheorem{theorem}{Theorem}

\newtheorem{question}[theorem]{Question}

\newtheorem{corollary}[theorem]{Corollary}
\newtheorem{lemma}[theorem]{Lemma}
\theoremstyle{definition}
\newtheorem{remark}[theorem]{Remark}

\newcommand{\N}{\mathbb N}

\def\N{\mathbb N}
\def\t_c{\tau_{comp}}
\def\op{\operatorname}

\begin{document}

\title[Embedding of graph inverse semigroups into CLP-compact topological semigroups]{Embedding of graph inverse semigroups into CLP-compact topological semigroups}

\author[S.~Bardyla]{Serhii~Bardyla}
\thanks{The work of the author is supported by the Austrian Science Fund FWF (Grant  I
3709 N35).}
\address{S. Bardyla: Institute of Mathematics, Kurt G\"{o}del Research Center, University of Vienna, Austria}
\email{sbardyla@yahoo.com}



\begin{abstract}
In this paper we investigate graph inverse semigroups which are subsemigroups of compact-like topological semigroups. More precisely, we characterise graph inverse semigroups which admit a compact semigroup topology and describe graph inverse semigroups which can be embeded densely into CLP-compact topological semigroups.
\end{abstract}
\maketitle

\section{Preliminaries}
In this paper all topological spaces are assumed to be Hausdorff. We shall follow the terminology of~\cite{Clifford-Preston-1961-1967,
Engelking-1989, Lawson-1998}. By $\omega$ we denote the first infinite ordinal. Put $\N=\omega\setminus\{0\}$. The cardinality of a set $X$ is denoted by $|X|$.

A semigroup $S$ is called an \emph{inverse semigroup} if for each element $a\in S$ there exists a unique
inverse element $a^{-1}\in S$ such that $aa^{-1}a=a$ and $a^{-1}aa^{-1}=a^{-1}$.
The map which associates every element of an inverse semigroup to its
inverse is called an \emph{inversion}.

For a subset $A$ of a topological space $X$ by
$\overline{A}^X$ (or simply $\overline{A}$)
we denote the closure of the set $A$ in $X$.
A topological space $X$ is said to be
\begin{itemize}
\item[$\bullet$] {\em compact}, if each open cover of $X$ contains a finite subcover;
\item[$\bullet$] {\em countably compact}, if each infinite subset $A\subset X$ has an accumulation point;
\item[$\bullet$] {\em countably compact at a subset} $A\subset X$, if each infinite subset $B\subset A$ has an accumulation point in $X$;
\item[$\bullet$] {\em countably pracompact}, if there exists a dense subset $A$ of $X$ such that $X$ is countably compact at $A$;
\item[$\bullet$] {\em pseudocompact}, if $X$ is Tychonoff and each continuous real-valued function on $X$ is bounded;
\item[$\bullet$] {\em feebly compact}, if each locally finite family of non-empty open subsets of $X$
  is finite;
\item[$\bullet$] {\em CLP-compact}, if every cover of $X$ consisting of clopen sets has a finite subcover.
\end{itemize}

By~\cite{G-R-2018}, for a topological space $X$ the following implications hold: $X$ is compact $\Rightarrow$ $X$ is countably compact $\Rightarrow$ $X$ is countably pracompact $\Rightarrow$ $X$ is feebly compact. Also, a feebly compact topological space $X$ is pseudocompact iff $X$ is Tychonoff. It is easy to check that each feebly compact space is CLP-compact. CLP-compact spaces were investigated in~\cite{Dik} and~\cite{Med}.

A {\em topological (inverse) semigroup} is a topological space together with a continuous semigroup operation (and an~inversion, respectively).  If $S$ is a semigroup (an inverse semigroup) and $\tau$ is a topology on $S$ such that $(S,\tau)$ is a topological (inverse) semigroup, then we
shall call $\tau$ a (\emph{inverse}) \emph{semigroup} \emph{topology} on $S$.
A {\em semitopological semigroup} is a topological space together with a separately continuous semigroup operation.

Let $X$ be a non-empty set. By $\mathcal{M}_{X}$ we denote the set
 $
 X{\times}X\cup\{0\}
 $
where $0\notin X{\times}X$ endowed with the following semigroup
operation:
\begin{equation*}
\begin{split}
&(a,b)\cdot(c,d)=
\left\{
  \begin{array}{cl}
    (a,d), & \hbox{ if~ } b=c;\\
    0, & \hbox{ if~ } b\neq c,
  \end{array}
\right.\\
&\hbox{and } (a,b)\cdot 0=0\cdot(a,b)=0\cdot 0=0, \hbox{ for each } a,b,c,d\in X.
\end{split}
\end{equation*}
The semigroup $\mathcal{M}_{X}$ is called the \emph{semigroup of $X{\times}X$-matrix units}. Observe that semigroups $\mathcal{M}_{X}$ and $\mathcal{M}_{Y}$ are isomorphic iff $|X|=|Y|$.

If a set $X$ is infinite, then the semigroup of $X{\times}X$-matrix units cannot be embedded into a compact topological semigroup (see~\cite[Theorem 3]{Gutik-2005}). In~\cite[Theorem~5]{Gutik-2009} the above result was extended over the class of countably compact topological semigroups.
In~\cite[Theorem~4.4]{BardGut-2016(1)} it was showed that for an infinite set $X$ the semigroup $\mathcal{M}_{X}$ cannot be embedded as a dense subsemigroup into a feebly compact topological semigroup.

A {\em bicyclic monoid} $\mathcal{C}(p,q)$ is the semigroup with the identity $1$ generated by two elements $p$ and $q$ subject to the condition $pq=1$.
Topologization of the bicyclic semigroup was investigated in~\cite{Eberhart-Selden-1969} and~\cite{Gutik-2015}. Neither stable nor $\Gamma$-compact topological semigroups can
contain a copy of the bicyclic semigroup (see \cite{Anderson-Hunter-Koch-1965,Hildebrant-Koch-1988}). In~\cite{GutRep-2007} it was proved that the bicyclic monoid does not embed into a countably compact topological inverse semigroup. A topological semigroup which has a pseudocompact square cannot contain the bicyclic monoid~\cite{BanDimGut-2010}. However, in~\cite[Theorem 6.6]{BanDimGut-2010} it was proved that under Martin's Axiom there exists a Tychonoff countably compact topological semigroup $S$ containing the bicyclic monoid.  The existence of a ZFC-example of a countably compact topological semigroup which contains the bicyclic monoid is still an open problem (see~\cite[Problem 7.1]{BanDimGut-2010}).

One of the generalizations of the bicyclic monoid is a polycyclic monoid.
For a cardinal $\lambda$ a {\em polycyclic monoid} $\mathcal{P}_\lambda$ is the semigroup with identity $1$ and zero $0$ given by the presentation:
\begin{equation*}
    \mathcal{P}_\lambda=\left\langle 0,1, \left\{p_i\right\}_{i\in\lambda}, \left\{p_i^{-1}\right\}_{i\in\lambda}\mid  p_i^{-1}p_i=1, p_j^{-1}p_i=0 \hbox{~for~} i\neq j\right\rangle.
\end{equation*}

Polycyclic monoid $\mathcal{P}_{k}$ over a finite cardinal $k\geq 2$ was introduced in \cite{Nivat-Perrot-1970}. Observe that the bicyclic semigroup with adjoined zero is isomorphic to the polycyclic monoid $\mathcal{P}_{1}$.
Embedding of the polycyclic monoid into compact-like topological semigroups were investigated in~\cite{BardGut-2016(2)}. More precisely, it was proved that for each cardinal $\lambda>1$ polycyclic monoid $\mathcal{P}_{\lambda}$ does not embed as a dense subsemigroup into a feebly compact topological semigroup.

A {\em directed graph} $E=(E^{0},E^{1},r,s)$ consists of disjoint sets $E^{0},E^{1}$ of {\em vertices} and {\em edges}, respectively, together with functions $s,r:E^{1}\rightarrow E^{0}$ which are called {\em source} and {\em range}, respectively. In this paper we refer to directed graph simply as ``graph". We consider each vertex being a path of length zero. A path of a non-zero length $x=e_{1}\ldots e_{n}$ in a graph $E$ is a finite sequence of edges $e_{1},\ldots,e_{n}$ such that $r(e_{i})=s(e_{i+1})$ for each positive integer $i<n$. By $\operatorname{Path}^+(E)$ we denote the set of all paths of a graph $E$ which have a non-zero length. We extend functions $s$ and $r$ on the set $\operatorname{Path}(E)=E^0\cup \operatorname{Path}^+(E)$ of all paths in graph $E$ as follows: for each vertex $e\in E^0$ put $s(e)=r(e)=e$ and for each path of a non-zero length $x=e_{1}\ldots e_{n}$ put $s(x)=s(e_{1})$ and $r(x)=r(e_{n})$. By $|x|$ we denote the length of a path $x$. Let $a=e_1\ldots e_n$ and $b=f_1\ldots f_m$ be two paths such that $|a|>0$, $|b|>0$ and $r(a)=s(b)$. Then by $ab$ we denote the path $e_1\ldots e_nf_1\ldots f_m$.
If $a$ is a vertex and $b$ is a path such that $s(b)=a$ ($r(b)=a$, resp.) then put $ab=b$ ($ba=b$, resp.).
An edge $e$ is called a {\em loop} if $s(e)=r(e)$. A path $x$ is called a {\em cycle} if $s(x)=r(x)$ and $|x|>0$. A graph $E$ is called {\em acyclic} if it contains no cycles.

For a given directed graph $E=(E^{0},E^{1},r,s)$ a graph inverse semigroup (or simply GIS) $G(E)$ over a graph $E$ is a semigroup with zero generated by the sets $E^{0}$, $E^{1}$ together with a set $E^{-1}=\{e^{-1}|\hbox{ } e\in E^{1}\}$ which is disjoint with $E^0\cup E^1$ satisfying the following relations for all $a,b\in E^{0}$ and $e,f\in E^{1}$:
 \begin{itemize}
 \item [(i)]  $a\cdot b=a$ if $a=b$ and $a\cdot b=0$ if $a\neq b$;
 \item [(ii)] $s(e)\cdot e=e\cdot r(e)=e;$
 \item [(iii)] $e^{-1}\cdot s(e)=r(e)\cdot e^{-1}=e^{-1};$
 \item [(iv)] $e^{-1}\cdot f=r(e)$ if $e=f$ and $e^{-1}\cdot f=0$ if $e\neq f$.
\end{itemize}

Graph inverse semigroups are generalizations of the polycyclic monoids. In particular, for every cardinal $\lambda$ polycyclic monoid $\mathcal{P}_{\lambda}$ is isomorphic to the graph inverse semigroup over the graph which consists of one vertex and $\lambda$ distinct loops. However, by~\cite[Theorem~1]{Bardyla-2017(2)}, each graph inverse semigroup $G(E)$ is isomorphic to a subsemigroup of the polycyclic monoid $\mathcal{P}_{|G(E)|}$.

According to~\cite[Chapter~3.1]{Jones-2011}, each non-zero element of a graph inverse semigroup $G(E)$ can be uniquely represented as $uv^{-1}$ where $u,v\in \operatorname{Path}(E)$ and $r(u)=r(v)$. A semigroup operation in $G(E)$ is defined by the following way:
\begin{equation*}
\begin{split}
  &  u_1v_1^{-1}\cdot u_2v_2^{-1}=
    \left\{
      \begin{array}{ccl}
        u_1wv_2^{-1}, & \hbox{if~~} u_2=v_1w & \hbox{for some~} w\in \operatorname{Path}(E);\\
        u_1(v_2w)^{-1},   & \hbox{if~~} v_1=u_2w & \hbox{for some~} w\in \operatorname{Path}(E);\\
        0,              & \hbox{otherwise},
      \end{array}
    \right.\\
  &  \hbox{and } uv^{-1}\cdot 0=0\cdot uv^{-1}=0\cdot 0=0.
    \end{split}
\end{equation*}

Further, when we write an element of $G(E)$ in a form $uv^{-1}$ we always mean that $u,v\in\operatorname{Path}(E)$ and $r(u)=r(v)$.
Simple verifications show that $G(E)$ is an inverse semigroup and $(uv^{-1})^{-1}=vu^{-1}$.

Graph inverse semigroups play an important role in the study of rings and $C^{*}$-algebras (see \cite{Abrams-2005,Ara-2007,Cuntz-1980,Kumjian-1998,Paterson-1999}).
Algebraic theory of graph inverse semigroups is well developed (see~\cite{Amal-2016, Bardyla-2018(1), Bardyla-2017(2), Jones-2011, Jones-Lawson-2014, Lawson-2009, Mesyan-2016}).
Topological properties of graph inverse semigroups were investigated
in~\cite{Bardyla-2016(1),Bardyla-2017(1),Bardyla-2018, BardGut-2016(1),Mesyan-Mitchell-Morayne-Peresse-2013}.

In this paper we investigate graph inverse semigroups which are subsemigroups of compact-like topological semigroups. More precisely, we characterise graph inverse semigroups which admit a compact semigroup topology and describe graph inverse semigroups which can be embedded densely into CLP-compact topological semigroups.

\section{Graph inverse semigroups which admit a compact semigroup topology}
For any sets $A,B$ define $A\subset^* B$ if $A\setminus B$ is finite. The following technical lemma will be very useful in this paper.

\begin{lemma}\label{lemma0}
Let $X$ be a CLP-compact space with a dense discrete subspace $Y$. Then for each subset $Z\subset^* Y$ the set $\overline{Z}$ is countably compact at $Z$. Moreover, if $\overline{Z}\setminus Z$ is compact then $\overline{Z}$ is compact.
\end{lemma}

\begin{proof}
Let $Y$ be a dense discrete subset of a CLP-compact space $X$ and $Z\subset^* Y$. To obtain a contradiction assume that there exists an infinite discrete subset $A\subset Z$ which is closed in $\overline{Z}$ (and hence in $X$). Since $Z\subset^* Y$ the set $A\cap Y$ is infinite. The discreteness of $A$ implies that the set $A\cap Y$ is closed in $X$. At this point it easy to see that the cover $\mathcal{U}=\{\{a\}\mid a\in A\}\cup \{X\setminus A\}$ of $X$ consists of clopen subsets and admits no finite subcover which contradicts CLP-compactness of $X$. Hence each infinite subset of $Z$ has an accumulation point in $\overline{Z}$ which implies that $\overline{Z}$ is countably compact at $Z$.

Assume that the set $\overline{Z}\setminus Z$ is compact. Since $Z\subset^*Y$, $Z= Z\cap Y\cup\{z_1,\ldots, z_k\}$. By $Z_d$ we denote the set of all isolated points of $Z$. Since $\overline{Z}$ is countably compact at the set $Z$ we obtain that for each infinite subset $A$ of $Z$ there exists an accumulation point $a\in \overline{Z}\setminus Z_d$.
Let $\mathcal{F}$ be an open cover of $\overline{Z}$. Since the set $\overline{Z}\setminus Z_d=\overline{Z}\setminus Z\cup\{z_1,\ldots,z_k\}$ is compact there exists a finite subset $\{F_1,\ldots, F_n\}\subset \mathcal{F}$ such that $\overline{Z}\setminus Z_d\subset \cup_{i=1}^n F_i$. We claim that the set $A=\overline{Z}\setminus (\cup_{i=1}^nF_i)$ is finite. Indeed, if the set $A$ is infinite, then it has an accumulation point $a\in \overline{Z}\setminus Z_d$. However, the set $\cup_{i=1}^nF_i$ is an open neighborhood of $a$ which does not intersect the set $A$. The obtained contradiction implies that the set $A$ is finite. Put $A=\{y_1,\ldots,y_m\}$. For each $i\leq m$ there exists an element $F_{y_i}\in \mathcal{F}$ such that $y_i\in F_{y_i}$. Then $\mathcal{G}=\{F_1,\ldots,F_n,F_{y_1},\ldots,F_{y_m}\}$ is a finite subcover of $\mathcal{F}$. Hence the space $\overline{Z}$ is compact.
\end{proof}



Let $G(E)$ be an arbitrary semitopological GIS. Observe that each non-zero element of $G(E)$ is isolated (see~\cite[Theorem~4]{Bardyla-2018}). For an arbitrary GIS $G(E)$ by $\tau_{c}$ we denote the topology on $G(E)$ which is defined as follows: each non-zero element is isolated in $(G(E),\tau_{c})$ and open neighborhood base of the point $0$ consists of cofinite subsets of $G(E)$ which contain $0$. According to~\cite[Lemma 3]{Bardyla-2018}, $\tau_{c}$ is the unique topology which makes $G(E)$ a compact semitopological semigroup.
Lemma~\ref{lemma0} implies the following two corollaries:

\begin{corollary}\label{lemma1}
For an arbitrary semitopological GIS $G(E)$ the following conditions are equivalent:
\begin{itemize}
\item[(1)] $G(E)$ is compact;
\item[(2)] $G(E)$ is countably compact;
\item[(3)] $G(E)$ is feebly compact;
\item[(4)] $G(E)$ is CLP-compact;
\item[(5)] $G(E)$ is topologically isomorphic to the $(G(E),\tau_{c})$.
\end{itemize}
\end{corollary}

\begin{corollary}\label{cor1}
Let $S$ be a CLP-compact semitopological semigroup which contains densely a graph inverse semigroup $G(E)$. Then for each
subset $X\subset G(E)$ the set $\overline{X}^S$ is countably compact at $X$.
\end{corollary}

By~\cite[Corollary~2]{Mesyan-2016}, two non-zero elements $ab^{-1}$ and $cd^{-1}$ of a GIS $G(E)$ are $\mathcal{D}$-equivalent iff $r(a)=r(b)=r(c)=r(d)$. For more about Green's relations on graph inverse semigroups see~\cite{Mesyan-2016}. Observe that each non-zero $\mathcal{D}$-class contains exactly one vertex of a graph $E$. By $D_e$ we denote the $\mathcal{D}$-class which contains vertex $e\in E^0$.
The following theorem characterises graph inverse semigroups which admit a compact (inverse) semigroup topology.

\begin{theorem}\label{th1}
Let $G(E)$ be a semitopological GIS. Then the following statements are equivalent:
\begin{itemize}
\item[$(1)$] the semigroup operation is jointly continuous in $(G(E),\tau_{c})$;
\item[$(2)$] for each element $uv^{-1}\in G(E)\setminus\{0\}$ the set $M_{uv^{-1}}=\{(ab^{-1},cd^{-1})\in G(E){\times}G(E)\mid ab^{-1}\cdot cd^{-1}=uv^{-1}\}$ is finite;
\item[$(3)$] for each vertex $e$ the set $I_e=\{u\in\op{Path}(E)\mid r(u)=e\}$ is finite;
\item[$(4)$] $G(E)$ neither contains the bicyclic monoid nor the semigroup of $\omega{\times}\omega$-matrix units;
\item[$(5)$] each $\mathcal{D}$-class in $G(E)$ is finite.
\end{itemize}
\end{theorem}

\begin{proof}
$(1)\Rightarrow(2)$. Suppose that $(G(E),\tau_{c})$ is a topological semigroup and fix an arbitrary non-zero element $uv^{-1}\in G(E)$. Observe that, by~\cite[Lemma 1]{Mesyan-Mitchell-Morayne-Peresse-2013}, for a fixed non-zero element $ab^{-1}\in G(E)$ the subsets
$$A_{ab^{-1}}=\{(ab^{-1},xy^{-1})\mid ab^{-1}\cdot xy^{-1}=uv^{-1}\}\subset G(E){\times}G(E)\quad \hbox{ and }$$
$$B_{ab^{-1}}=\{(xy^{-1},ab^{-1})\mid xy^{-1}\cdot ab^{-1}=uv^{-1}\}\subset G(E){\times}G(E)$$ are finite. The continuity of the semigroup operation in $(G(E),\tau_c)$ yields an open neighborhood $V=G(E)\setminus\{x_1,\ldots, x_n\}$ of $0$ such that $V\cdot V\subseteq G(E)\setminus\{uv^{-1}\}$. Observe that $M_{uv^{-1}}\subset (\cup_{i=1}^n (A_{x_i}\cup B_{x_i}))\cup(V{\times} V)$. Since $uv^{-1}\notin V\cdot V$ we obtain that $M_{uv^{-1}}\cap (V{\times} V)=\emptyset$. Since the sets $A_{x_i}$ and $B_{x_i}$ are finite the set $M_{uv^{-1}}$ is finite as well.

$(2)\Rightarrow (3)$. Suppose that there exists a vertex $e\in E^0$ such that the set $I_e=\{x\in \op{Path}(E)\mid r(x)=e\}$ is infinite. Then $\{(x^{-1},x)\mid x\in I_e\}$ is an infinite subset of $M_{e}$ which contradicts condition $(2)$.

$(3)\Rightarrow (4)$. Suppose that $G(E)$ contains an isomorphic copy of the bicyclic monoid. Then there exists an element $uv^{-1}\in G(E)$ such that $uv^{-1}\cdot uv^{-1}\notin \{0,uv^{-1}\}$. The definition of the semigroup operation in $G(E)$ implies that either $v=uw$ or $u=vw$. Then in both cases $s(w)=r(u)=r(v)$ and $r(w)=r(u)=r(v)$ which implies that $w$ is a cycle. Then the set $I_{r(w)}=\{u\in \op{Path}(E)\mid r(u)=r(w)\}$ is infinite, because it contains the set $\{w^n\mid n\in\N\}$, which contradicts condition $(3)$.

Suppose that $G(E)$ contains the semigroup of $\omega{\times}\omega$-matrix units $\mathcal{M}_{\omega}$. Observe that for each non-zero idempotent $e\in \mathcal{M}_{\omega}$ the set $\{x\in \mathcal{M}_{\omega}\mid x\cdot x^{-1}=e\}$ is infinite. Then there exists an idempotent $vv^{-1}\in G(E)\setminus\{0\}$ and an infinite subset $A=\{a_ib_i^{-1}\mid i\in\omega\}$ of $G(E)\setminus\{0\}$ such that
$$a_ib_i^{-1}\cdot (a_ib_i^{-1})^{-1}=a_ib_i^{-1}\cdot b_ia_i^{-1}=a_ia_i^{-1}=vv^{-1}$$
for each $i\in\omega$.
The above equation implies that $a_i=v$ for each $i\in\omega$. Since the set $A$ is infinite we obtain that the set $B=\{b_i\mid i\in \omega\}$ is infinite. Observe that $r(b_i)=r(a_i)=r(v)$ for each $i\in\omega$. Then the set $I_{r(v)}$ is infinite, because it contains the set $B$, which contradicts condition $(3)$.

$(4)\Rightarrow (5)$.
Suppose that there exists a vertex $e$ such that the $\mathcal{D}$-class $D_e$ is infinite. Then one of the following two cases holds:
\begin{itemize}
\item[$(i)$] there exists a cycle $u$ such that $s(u)=r(u)=e$;
\item[$(ii)$] The set $I_e$ contains no cycles.
\end{itemize}
Consider case $(i)$. Fix an arbitrary cycle $u\in D_e$. By $S$ we denote the subsemigroup of $G(E)$ which is generated by two elements $u$ and $u^{-1}$. Routine verifications show that the semigroup $S$ is isomorphic to the bicyclic monoid ($e=u^{-1}u$ is the identity of $S$) which contradicts condition $(4)$.

Consider case $(ii)$. By~\cite[Corollary~5]{Bardyla-2018(1)}, $D_e\cup\{0\}$ is a subsemigroup of $G(E)$ which is isomorphic to the semigroup of
$|I_e|{\times}|I_e|$-matrix units. Observe that $|I_e|$ is an infinite cardinal. Since for each cardinals $k\leq t$ the semigroup $\mathcal{M}_{k}$ is a subsemigroup of $\mathcal{M}_{t}$ we obtain that $G(E)$ contains the semigroup $\mathcal{M}_{\omega}$ which contradicts condition $(4)$.

$(5)\Rightarrow (1)$. Suppose that each $\mathcal{D}$-class of a GIS $G(E)$ is finite. Then the graph $E$ is acyclic. By~\cite[Lemma 3]{Bardyla-2018}, $(G(E),\tau_{c})$ is a semitopological semigroup. Since each non-zero element of $(G(E),\tau_{c})$ is isolated the semigroup operation in $(G(E),\tau_{c})$ is continuous if it is continuous at the point $(0,0)\in G(E){\times}G(E)$. Fix an arbitrary open neighborhood $U$ of $0$. Let $G(E)\setminus U=\{u_1v_1^{-1},\ldots, u_nv_n^{-1}\}$. Put $V=G(E)\setminus (\cup_{i=1}^nD_{r(u_i)})$. Since each $\mathcal{D}$-class is finite $V$ is an open neighborhood of $0$. The inclusion $V\cdot V\subseteq U$ follows from~\cite[Lemma~1]{Bardyla-2018(1)} which states that $(D_e\cup \{0\})\cdot (D_f\cup\{0\})\subseteq D_e\cup D_f\cup\{0\}$. Hence $(G(E),\tau_{c})$ is a topological semigroup.
\end{proof}

For an arbitrary semigroup $S$ by $E(S)$ we denote the set of all idempotents of $S$. If $S$ is an inverse semigroup, then $E(S)$ is a commutative subsemigroup of $S$. Moreover, $E(S)$ is a semilattice with respect to the following natural partial order: $e\leq f$ iff $e=e\cdot f=f\cdot e$, for $e,f\in E(S)$.

\begin{lemma}\label{lemmacomp}
Let $G(E)$ be a topological GIS such that $E(G(E))$ is compact. Then $G(E)$ is closed in each topological semigroup $S$ which contains $G(E)$ as a subsemigroup.
\end{lemma}

\begin{proof}
Assuming the contrary let $S$ be a topological semigroup which contains $G(E)$ as a dense proper subsemigroup. Fix an arbitrary $s\in S\setminus G(E)$ and disjoint open neighborhoods $U(s)$ and $U(0)$ of $s$ and $0$, respectively. The continuity of the semigroup operation in $S$ yields that $s\cdot 0=0\cdot s=0$ for each $s\in S$. Since $S$ is a topological semigroup there exist open neighborhoods $V(s)$ and $V(0)$ of $s$ and $0$, respectively, such that $V(s)\subseteq U(s)$, $V(0)\subseteq U(0)$ and $V(s)\cdot V(0)\cup V(0)\cdot V(s)\subseteq U(0)$.
Since the set $V(s)\cap G(E)$ is infinite we obtain that the set
$$A=\{u\in \op{Path}(E)\mid \hbox{ there exists a path }v \hbox{ such that } uv^{-1}\in V(s)\cap G(E)\hbox{ or } vu^{-1}\in V(s)\cap G(E) \}$$
is infinite as well. By the compactness of $E(G(E))$, the set $E(G(E))\setminus V(0)$ is finite. Then there exist elements $u,v\in \op{Path}(E)$ such that $uu^{-1}\in V(0)$ and either $uv^{-1}\in V(s)$ or $vu^{-1}\in V(s)$. Hence either $uv^{-1}=uu^{-1}\cdot uv^{-1}\in V(0)\cdot V(s)\subseteq U(0)$ or $vu^{-1}=vu^{-1}\cdot uu^{-1}\in V(s)\cdot V(0)\subseteq U(0)$ which contradicts to the choice of the sets $U(0)$ and $U(s)$.
\end{proof}

\section{Main theorem}
By $E_1\sqcup E_2$ we denote the disjoint union of graphs $E_1$ and $E_2$.
Recall that for a fixed vertex $e\in E^0$, $I_e=\{u\in \op{Path}(E)\mid r(u)=e\}$.
Now we are going to formulate the main theorem of this paper.

\begin{theorem}\label{main}
Let a graph inverse semigroup $G(E)$ be a dense subsemigroup of a CLP-compact topological semigroup $S$. Then the following statements hold:
\begin{itemize}
\item[$(1)$] there exists a cardinal $k$ such that $E=(\sqcup_{\alpha\in k} E_{\alpha})\sqcup F$ where the graph
$F$ is acyclic and
for each $\alpha\in k$ the graph $E_{\alpha}$ consists of one vertex and one loop;
\item[$(2)$] if the graph $F$ is non-empty, then for each vertex $f\in F^0$ the set $I_f$ is finite and the semigroup $G(F)$ is a compact subset of $G(E)$;
\item[$(3)$] each open neighborhood of $0$ contains all but finitely many subsets $G(E_{\alpha})\subset G(E)$, $\alpha\in k$.
\end{itemize}
\end{theorem}

The proof of Theorem~\ref{main} could be found in Section~\ref{proof} after some preparatory work made in Section~\ref{acycl}

\section{Embedding of graph inverse semigroups over acyclic graphs into CLP-compact topological semigroups}\label{acycl}

We start with a formulation of the main result of this section which helps us to prove statement $(2)$ of Theorem~\ref{main}.

\begin{theorem}\label{th2}
Let $E$ be an acyclic graph. Then GIS $G(E)$ embeds as a dense subsemigroup into a CLP-compact topological semigroup $S$ iff $G(E)$ is compact, i.e., $G(E)=(G(E),\tau_{c})$.
\end{theorem}

Proof of Theorem~\ref{th2} could be found at the end of this section after some preparatory work made in lemmas~\ref{l0}-\ref{thT}.

\begin{lemma}\label{l0}
Let $E$ be an acyclic graph and $G(E)$ be a dense subsemigroup of a topological semigroup $S$. Then $s^2=0$ for each non-idempotent element $s\in S$.
\end{lemma}
\begin{proof}
By~\cite[Corollary~5]{Bardyla-2018(1)}, for each $e\in E^0$ the set $D_e\cup\{0\}$ is isomorphic to the semigroup of $I_e{\times}I_e$-matrix units. Hence $x\cdot x=0$ for each non-idempotent element $x\in G(E)$.
Fix an arbitrary element $s\in S\setminus E(S)$.
Since $E(S)$ is closed in $S$ and $G(E)$ is dense in $S$ we obtain that $V\cap (G(E)\setminus E(S))\neq\emptyset$ for each open neighborhood $V$  of $s$. Hence $0\in V^2$ for each open neighborhood $V$ of $s$ witnessing that $s^2=0$.
\end{proof}

\begin{lemma}\label{c1}
Let $E$ be an acyclic graph and $G(E)$ be a dense subsemigroup of a topological semigroup $S$. Then $E(S)=\overline{E(G(E))}$ and the set $E(S)\setminus \{0\}$ is open in $S$;
\end{lemma}
\begin{proof}
Put $A=\{x\in S\mid x^2=0\}$.
By Lemma~\ref{l0}, $E(S)\setminus\{0\}=S\setminus A$. Since the map $f:S\rightarrow S$, $f(s)=s^2$ is continuous, the set $A$ is closed. Hence the set $E(S)\setminus \{0\}$ is open in $S$. Fix an arbitrary $s\in E(S)\setminus E(G(E))$. Since $G(E)$ is dense in $S$ and $s\notin A$ each open neighborhood of $s$ contains infinitely many $x\in G(E)\setminus A\subset E(G(E))$. Hence $s\in \overline{E(G(E))}$.
\end{proof}




Further we shall need properties of the graph inverse semigroup $G(T)$ where by $T$ we denote the unary tree (see picture below).

\centerline{
\xymatrix{
    \bullet_{0} \ar[r]^{(0,1)}&\bullet_{1} \ar[r]^{(1,2)}&\bullet_{2} \ar[r]^{(2,3)}& \bullet_{3} \cdots &\bullet_{n} \ar[r]^{(n,n+1)}& \bullet_{n+1} \cdots
    }
}
We enumerate vertices of the tree $T$ with non-negative integers and identify each edge $x$ with a pair of integers $(n,n+1)$ where $s(x)=n$ and $r(x)=n+1$.
For each positive integers $k\leq p$ by $(k,p)$ we denote the path $u$ such that $s(u)=k$ and $r(u)=p$ (for convenience we identify vertex $n$ with the pair $(n,n)$).
Observe that $(n,m)(n,m)^{-1}\leq(k,l)(k,l)^{-1}$ iff $n=k$ and $m\geq l$. Simply verifications show that each maximal chain (linear ordered subset) in $E(G(T))$ coincides with the set $L_n=\{(n,m)(n,m)^{-1}\mid m\geq n\}{\cup}\{0\}$ for some fixed $n\in \omega$. By $L_n^*$ we denote the set $L_n{\setminus}\{0\}$.

\begin{lemma}\label{TT}
Let $G(E)$ be a GIS over an acyclic graph, and $C\subset E(G(E))$ be an infinite chain. Then there exists a subgraph $T\subset E$ which is isomorphic to the unary three and $C\subset E(G(T))$.
\end{lemma}

\begin{proof}
Using the axiom of choice we can find a maximal chain $L$ which contains $C$. Easy to see that $L=\{u_nu_n^{-1}\}_{n\in\omega}\cup\{0\}$ where $u_0$ is some vertex and for each $n\in \omega$ there exists an edge $x_n\in E^1$ such that $u_{n+1}=u_nx_n$. Since a graph $E$ is acyclic for each $n\in\omega$ neither $u_n$ nor $x_n$ is a cycle. Then $T=(T^0, T^1,s_T,r_T)$ is a desired unary tree, where $T^0=\{r(u_n)\mid n\in\omega\}$, $T^1=\{x_n\mid n\in\omega\}$ and the function $s_T$ ($r_T$, resp.) is a restriction of the source (range, resp.) function $s$ ($r$, resp.) of the graph $E$ on the set $T^1$.
\end{proof}

\begin{lemma}\label{lemma4}
Let $G(T)$ be a semitopological semigroup.
If there exists $k\in \omega$ such that $0$ is not an accumulation point of the set $L_{k}^*$, then $0$ is not an accumulation point of the set $L_n^*$ for each $n\in\omega$.
\end{lemma}

\begin{proof}
Let $k$ be a non-negative integer such that $0\notin \overline{L_k^*}$.
Suppose to the contrary that $0\in \overline{L_n^*}$ for some $n\in\omega$.
There are two cases to consider:
\begin{itemize}
\item[$(1)$] $k<n$;
\item[$(2)$] $n<k$.
\end{itemize}
Consider case $(1)$. Fix an arbitrary open neighborhood $U$ of $0$ such that $U\cap L_k^*=\emptyset$. Since $(k,n)\cdot 0\cdot (k,n)^{-1}=0$ the continuity of left and right shifts in $G(T)$ yields an open neighborhood $V$ of $0$ such that $(k,n)\cdot V\cdot (k,n)^{-1}\subseteq U$. Since $0\in \overline{L_n^*}$ we obtain that the set $V$ contains an element $(n,m)(n,m)^{-1}$ for some $m>n$. Hence $$(k,n)\cdot (n,m)(n,m)^{-1}\cdot (k,n)^{-1}=(k,m)(k,m)^{-1}\in U\cap L_k^*$$
 which contradicts to the choice of the set $U$.

To obtain a contradiction in case $(2)$ we consider the product $(n,k)^{-1}\cdot 0\cdot (n,k)=0$. Fix open neighborhoods $U$ and $V$ of $0$ such that $U$ does not intersect the set $L_k^*$ and $(n,k)^{-1}\cdot V\cdot (n,k)\subseteq U$. Since $0\in \overline{L_n^*}$ there exists an element $(n,m)(n,m)^{-1}\in V$ such that $m>k$. Then $$(n,k)^{-1}\cdot (n,m)(n,m)^{-1}\cdot (n,k)=(k,m)(k,m)^{-1}\in U\cap L_k^*$$ which yields a contradiction.
\end{proof}

\begin{lemma}\label{lemmaT}
Let $G(T)$ be a dense subsemigroup of a CLP-compact topological semigroup $S$. Then for each open neighborhood $U$ of $0$ there exists $n\in \omega$ such that $L_k\subset U$ for each $k>n$.
\end{lemma}

\begin{proof}
Suppose to the contrary that there exists an open neighborhood $U$ of $0$ such that the set $A=\{k\in\omega \mid L_k\setminus U\neq \emptyset\}$ is infinite. For each $k\in A$ fix any idempotent $x_k\in L_k\setminus U$. By Lemma~\ref{c1}, $E(S)=\overline{E(G(T))}$. Corollary~\ref{cor1} implies that the set $E(S)$ is countably compact at $E(G(T))$. Hence there exists an element $s\in E(S)$ which is an accumulation point of the set $\{x_k\mid k\in A\}$. Since $U\cap\{x_k\mid k\in A\}=\emptyset$, $s\neq 0$. On the other hand, $x_k\cdot x_m=0$, whenever $k\neq m$. Hence $0\in V^2$ for each open neighborhood $V$ of $s$ witnessing that $s=0$. The contradiction.
\end{proof}

\begin{lemma}\label{thT}
$G(T)$ embeds as a dense subsemigroup into a CLP-compact topological semigroup $S$ iff $G(T)$ is compact, i.e., $G(T)=(G(T),\tau_{c})$.
\end{lemma}

\begin{proof}
Suppose that $G(T)$ is a dense subsemigroup of a CLP-compact topological semigroup $S$. There are two cases to consider:
\begin{itemize}
\item[$(1)$] there exists $n\in \omega$ such that $0$ is an accumulation point of $L_n^*$;
\item[$(2)$] $0$ is not an accumulation point of $L_n^*$ for each $n\in \omega$;
\end{itemize}

Consider case $(1)$. By Lemma~\ref{lemma4}, $0$ is an accumulation point of $L_n^*$ for each $n\in \omega$. Observe that for an arbitrary $n\in\omega$, $L_n^*$ is isomorphic to the semilattice $(\omega,\max)$.  By~\cite[Theorem~2]{Bardyla-Gutik-2012}, zero is a unique accumulation point of $L_n^*$ in $S$ for each $n\in\omega$. By Lemma~\ref{lemma0}, the set $L_n=\overline{L_n^*}$ is compact. Lemma~\ref{lemmaT} provides that the semilattice $E(G(T))=\cup_{n\in\omega}L_n$ is compact. Lemma~\ref{lemmacomp} implies that $G(T)$ is closed in $S$ and hence $S=G(T)$. By Corollary~\ref{lemma1}, $G(T)$ is compact.

Consider case $(2)$. Running ahead, we will show that this case is not possible. By Corollary~\ref{cor1}, the set $E(S)=\overline{E(G(T))}$ is countably compact at $E(G(T))$. By~\cite[Theorem~2]{Bardyla-Gutik-2012}, for each non-negative integer $n$, $L_n^*$ has only one accumulation point which we denote by $s_n$. Lemma~\ref{lemma0} implies that the set $L_n^*\cup\{s_n\}$ is compact.

Consider the sets $B^{+}=\{(0,m)\mid m\in \N\}$ and $B^{-}=\{(0,m)^{-1}\mid m\in \N\}$. By Lemma~\ref{lemma0}, the set $\overline{B^+}$ ($\overline{B^-}$, resp.) is countably compact at the set $B^+$ ($B^-$, resp.).
There are three subcases to consider:
\begin{itemize}
\item[$(2.1)$] $(\overline{B^{+}}\cup \overline{B^{-}})\cap (S\setminus G(E))=\emptyset$;
\item[$(2.2)$] $\overline{B^{+}}\cap (S\setminus G(E))\neq \emptyset$;
\item[$(2.3)$] $\overline{B^{-}}\cap (S\setminus G(E))\neq \emptyset$.
\end{itemize}
Consider subcase $(2.1)$. In this subcase the only accumulation point of the sets $\overline{B^{+}}$ and $\overline{B^{-}}$ is $0$. Lemma~\ref{lemma0} implies that each open neighborhood of $0$ contains all but finitely many elements of the set $B^{+}\cup B^-$. The continuity of the semigroup operation yields an open neighborhood $V$ of $0$ such that $V\cdot V\subset S\setminus\{(0,0)\}$. However, there exists a positive integer $n\in \N$ such that $(0,n)^{-1}\in V$ and $(0,n)\in V$ which implies that $(0,0)=(0,n)^{-1}\cdot (0,n)\in V\cdot V\subset S\setminus\{(0,0)\}$. The obtained contradiction implies that subcase $(2.1)$ does not hold.

Consider subcase $(2.2)$.
Fix an arbitrary element $x\in \overline{B^+}\setminus G(E)$. Since $s_0=\lim_{m\in\omega}(0,m)(0,m)^{-1}$ and $(0,m)(0,m)^{-1}\cdot (0,m)=(0,m)$ the continuity of the semigroup operation in $S$ implies that $s_0\cdot x=x$. Fix an arbitrary open neighborhood $U(x)$ of $x$. Since $s_0\cdot x=x$ the continuity of the semigroup operation in $S$ yields open neighborhoods $U(s_0)$ and $V(x)$ of $s_0$ and $x$, respectively, such that $U(s_0)\cdot V(x)\subset U(x)$. Observe that there exists $k\in \omega$ such that for each $n>k$ $(0,n)(0,n)^{-1}\in V(s_0)$. Since $x\in \overline{B^+}$ there exists an infinite subset $C\in\omega$ such that $(0,m)\in V(x)$ for each $m\in C$. Hence
$$(0,n)(m,n)^{-1}=(0,n)(0,n)^{-1}\cdot (0,m)\in U(s_0)\cdot V(x) \subset U(x),\hbox{  where } k<m<n \hbox{ and } m\in C.$$

Fix an arbitrary disjoint open neighborhoods $U(0)$ and $U(x)$ of $0$ and $x$, respectively. Since $x\cdot 0=0$ the continuity of the semigroup operation in $S$ yields open neighborhoods $V(0)\subset U(0)$ and $V(x)\subset U(x)$ of $0$ and $x$, respectively, such that $V(x)\cdot V(0)\subset U(0)$. By Lemma~\ref{lemmaT}, there exists $k\in \omega$ such that $L_n^*\subset V(0)$ for each $n>k$. Hence there exist integers $m<n\in\omega$ such that $(m,n)(m,n)^{-1}\in V(0)$ and $(0,n)(m,n)^{-1}\in V(x)$. Then $$(0,n)(m,n)^{-1}=(0,n)(m,n)^{-1}\cdot(m,n)(m,n)^{-1} \in V(x)\cdot V(0)\subset U(0)$$
which contradicts to the choice of the sets $U(0)$ and $U(x)$. The obtained contradiction implies that subcase $(2.2)$ does not hold.

Similarly, it could be shown that subcase $(2.3)$ does not hold as well. Hence case $(2)$ is impossible.
\end{proof}


\begin{center}
Proof of Theorem~\ref{th2}
\end{center}

Let $S$ be a CLP-compact topological semigroup which contains a graph inverse semigroup $G(E)$ over an acyclic graph $E$ as a dense subsemigroup. There are two cases to consider:
\begin{itemize}
\item[$(1)$] $E(S)=E(G(E))$;
\item[$(2)$] there exists an element $s\in E(S)\setminus E(G(E))$.
\end{itemize}
Consider case $(1)$. By Corollary~\ref{cor1}, the set $E(S)=\overline{E(G(E))}$ is countably compact at $E(G(E))$. Hence $E(S)$ is countably compact and the only non-isolated point in $E(S)$ could be $0$.
Lemma~\ref{lemma0} implies that the set $E(G(E))$ is compact.
By Lemma~\ref{lemmacomp}, the set $G(E)$ is closed in $S$. Hence $G(E)=S$. By Corollary~\ref{lemma1}, the set $G(E)$ is compact and $G(E)=(G(E),\tau_c)$.

Consider case $(2)$. Running ahead, we will show that this case is not possible.
Fix an arbitrary element $s\in E(S)\setminus E(G(E))$. By Lemma~\ref{c1}, there exists an open neighborhood $U$ of $s$ such that $U\cap G(E)\subset E(G(E))$. Since for each elements $a,b\in E(G(E))$, $a\cdot b \in \{a,b,0\}$ there exists a linearly ordered set $L\subset E(G(E))$ and an open neighborhood $V\subset U$ of $s$ such that $V\cap G(E)\subset L$ (in the other case each open neighborhood $W$ of $s$ will contain elements $a,b\in E(G(E))$ such that $a\cdot b=0$ and hence $0\in W^2$ which contradicts to the Hausdorffness of $S$). By Lemma~\ref{TT}, there exists an unary tree $T\subset E$ such that $L\subset G(T)\subset G(E)$. Observe that $s\in \overline{G(T)}$. Corollary~\ref{cor1} implies that $\overline{G(T)}$ is countably compact at $G(T)$. Hence $\overline{G(T)}$ is countably pracompact. Since each countably pracompact topological space is CLP-compact, Lemma~\ref{thT} implies that $G(T)$ is compact. Hence $\overline{G(T)}=G(T)$ which contradicts to the choice of $s$.

The following remark shows that Theorem~\ref{th2} does not hold for graphs which contain a cycle.

\begin{remark}\label{rem}
In~\cite[Theorem 6.1]{BanDimGut-2010} it was proved that there exists a Tychonoff countably pracompact topological semigroup $S$ which contains the bicyclic monoid as a dense subsemigroup. Moreover, under Martin's Axiom the semigroup $S$ is countably compact (see~\cite[Theorem 6.6 and Corollary 6.7]{BanDimGut-2010}).
Simply verifications show that the semigroup $S$ with adjoint isolated zero is countably pracompact and contains densely the discrete polycyclic monoid $\mathcal{P}_1$.
\end{remark}

Now we are able to prove our main result.

\section{Proof of Theorem~\ref{main}}\label{proof}

\begin{proof}
Assume that $G(E)$ is a dense subsemigroup of a CLP-compact topological semigroup $S$.
Put $$A=\{e\in E^0\mid \hbox { there exists a cycle } u \hbox{ such that } s(u)=r(u)=e\}.$$

Let $\kappa=|A|$. If $\kappa=0$ then the graph $E$ is acyclic and the proof follows from Theorem~\ref{th2} and Theorem~\ref{th1}.
Assume that $\kappa>0$. Fix an arbitrary vertex $e\in A$ and a cycle $u\in \op{Path}(E)$ such that $s(u)=r(u)=e$. We claim that for each vertex $f\neq e$ both sets
$$B_e=\{x\in E^1\mid s(x)=e \hbox{ and } r(x)=f\}\qquad \hbox{and} \qquad C_e=\{x\in E^1\mid s(x)=f \hbox{ and } r(x)=e\}$$
are empty. Indeed, suppose that there exists $x\in B$. Let $T$ be the inverse subsemigroup of $G(E)$ which is generated by the set $\{u^nx\mid n\in \N\}$. Observe that each non-zero element of $T$ is of the form $u^nx(u^mx)^{-1}$ and
\begin{equation*}
u^nx(u^mx)^{-1}\cdot u^kx(u^lx)^{-1}=
\left\{
  \begin{array}{cl}
    u^nx(u^lx)^{-1}, & \hbox{ if~ } m=k;\\
    0, & \hbox{ if~ } m\neq k.
  \end{array}
\right.
\end{equation*}
Simple verifications show that the semigroup $T$ is isomorphic to the semigroup of $\N{\times}\N$ matrix units (the isomorphism $h:T\rightarrow \mathcal{M}_{\N}$ can be defined as follows: $h(u^nx(u^mx)^{-1})=(n,m)$ and $h(0)=0$). Corollary~\ref{cor1} implies that $\overline{T}$ is countably compact at the dense subspace $T$. Hence $\overline{T}$ is countably pracompact. Since each countably pracompact topological space is feebly compact we obtain that $\overline{T}$ is feebly compact. This contradicts~\cite[Theorem~4.4]{BardGut-2016(1)} where it was proved that an infinite semigroup of matrix units cannot be embedded densely into a feebly compact topological semigroup. Hence the set $B_e$ is empty.

Assume that there exists $x\in C$. Let $T$ be the inverse subsemigroup of $G(E)$ which is generated by the set $\{xu^n\mid n\in \N\}$. Similar arguments imply that the semigroup $T$ is isomorphic to the semigroup of $\N{\times}\N$-matrix units and $\overline{T}$ is feebly compact. This contradicts~\cite[Theorem~4.4]{BardGut-2016(1)}. Hence the set $C_e$ is empty as well.

For each $x\in A$ put $E_{x}^0=\{x\}$, $E_x^1=\{y\in E^1\mid s(y)=r(y)=x\}$. Since the sets $B_x$ and $C_x$ are empty any cycle $u$ such that $s(u)=r(u)=x$ is a product of finitely many loops. Hence the set $E_x^1$ is non-empty. By $E_x$ we denote a subgraph of $E$ which contains one vertex $x$ and $E_x^1$ is the set of edges of $E_x$. Assume that there exists a vertex $x\in A$ and two distinct loops $y,z\in E_x^1$. Then, by~\cite[Theorem~3]{Bardyla-2018}, GIS $G(E)$ contains the polycyclic monoid $\mathcal{P}_2$ which is generated by the elements $y,z,y^{-1},z^{-1}$. By~\cite[Theorem~4.6]{BardGut-2016(1)}, monoid $\mathcal{P}_2$ cannot embed as a dense subset into a feebly compact topological semigroup. However, by Corollary~\ref{cor1}, the semigroup $\overline{\mathcal{P}_2}\subset S$ is feebly compact which implies the contradiction. Hence for each vertex $x\in A$ the set $E_x^1$ is singleton.

Put $F^0=\{f\in E^0\mid \hbox{ there exists no cycle } u \hbox{ such that } s(u)=r(u)=f\}$.
Denote $F^1=\{x\in E^1\mid s(x)\in F^0 \hbox{ and } r(x)\in F^0\}$.
Let $F=(F^0,F^1,s_F,r_F)$ be a subgraph of $E$ where $s_F$ ($r_F$, resp.) is the restriction of the source (range, resp.) function $s$ ($r$, resp.) of the graph $E$ on the set $F^1$. Observe that the graph $F$ is acyclic and $E=\sqcup_{\alpha\in \kappa}E_{\alpha}\sqcup F$.
If the graph $F$ is non-empty, then Corollary~\ref{cor1} implies that $\overline{G(F)}$ is a CLP-compact subsemigroup of $S$. Theorem~\ref{th2} implies that $G(F)$ is compact. By Theorem~\ref{th1}, the set $I_f$ is finite for each vertex $f\in F^0$.


This completes the proof of the statements $(1)$ and $(2)$.

Consider statement $(3)$. First we show that each open neighborhood $U$ of $0$ contains all but finitely many sets $E(G(E_{\alpha}))$, $\alpha\in k$.  Suppose to the contrary that there exist an open neighborhood $U$ of $0$ and an infinite subset $Z$ of $\kappa$ such that $E(G(E_{\alpha}))\setminus U\neq \emptyset$ for each $\alpha\in Z$. Fix an arbitrary element $x_{\alpha}\in E(G(E_{\alpha}))\setminus U$, $\alpha\in Z$. By Corollary~\ref{cor1}, the set $X=\{x_{\alpha}\mid \alpha\in Z\}$ has an accumulation point $y\in E(S)$. The continuity of the semigroup operation in $S$ yields an open neighborhood $U$ of $y$ such that $U\cdot U\subset S\setminus\{0\}$. Since $x_{\alpha}\cdot x_{\beta}=0$ if $\alpha\neq \beta$ we obtain that $0\in U\cdot U\subset S\setminus\{0\}$ which provides the contradiction.

Now we are ready to prove statement $(3)$. Suppose the contrary that there exists an open neighborhood $U$ of $0$ such that the set $Z=\{\alpha\in \kappa\mid G(E_{\alpha})\setminus U\}$ is infinite. For each $\alpha\in Z$ fix an arbitrary element $x_{\alpha}\in G(E_{\alpha})\setminus U$. Since the set $X=\{x_{\alpha}\mid \alpha\in Z\}$ is infinite and discrete, by Lemma~\ref{lemma0}, it has an accumulation point $s\in S$. By our assumption $s\neq 0$. Fix disjoint open neighborhoods $U(s)$ and $U(0)$ of points $s$ and $0$, respectively. Since $s\cdot 0=0$ the continuity of the semigroup operation in $S$ yields open neighborhoods $V(s)\subset U(s)$ and $V(0)\subset U(0)$ of points $s$ and $0$, respectively, such that $V(s)\cdot V(0)\subset U(0)$. Since $V(0)$ contains all but finitely many sets $E(G(E_{\alpha}))$, $\alpha\in\kappa$ and $s$ is an accumulation point of the infinite set $X$ we obtain that there exists element $x\in X\cap V(s)$ such that $x^{-1}x\in V(0)$. Hence $x=x\cdot x^{-1}x\in V(s)\cdot V(0)\subset U(0)$ which implies that $U(0)\cap U(s)\neq\emptyset$. The obtained contradiction finishes the proof of statement $(3)$.
\end{proof}

Since compact topological semigroups do not contain the bicyclic monoid, Theorem~\ref{main} implies the following:

\begin{corollary}
A GIS $G(E)$ embeds into a compact topological semigroup $S$ iff $G(E)$ is compact.
\end{corollary}

Unfortunately Theorem~\ref{main} says nothing about a topology on the subspaces $G(E_{\alpha})$ which are isomorphic to the bicyclic monoid with adjoint zero.
Remark~\ref{rem} implies that the bicyclic monoid with adjoint isolated zero can be embedded densely into a countably pracompact topological semigroup. However, the author does not know the answer to the following question:

\begin{question}
Can a bicyclic monoid with adjoint non-isolated zero be embedded into a countably compact topological semigroup?
\end{question}





\end{document}